\documentclass[12pt,leqno]{article}

\usepackage[lite]{amsrefs}
\usepackage{latexsym,enumerate}

\usepackage[notref, notcite]{showkeys}
\usepackage{amssymb, amsfonts, eucal,amsmath,amsthm}

\usepackage{color}
\definecolor{added}{rgb}{0, 0, 1}
\definecolor{deleted}{rgb}{1, 0, 0}

\newtheorem{theorem}{Theorem}[section]

\newtheorem{lemma}[theorem]{Lemma}

\newtheorem{remark}[theorem]{Remark}
\newtheorem{definition}[theorem]{Definition}

\newtheorem{corollary}[theorem]{Corollary}

\newcommand{\sect}[1]{\section{#1} \setcounter{equation}{0} }

\newcounter{ca}
\newcommand{\norm}[2]{\left\|#1\right\|_{#2}}

\newcommand{\Pn}{\mathbb P_n}

 \newcommand{\ec}{\end{comment}}
\newcommand{\bc}{ \begin{comment}
 }

\newcommand{\andd}{\quad\mbox{\rm and}\quad}

\newcommand\e{{\varepsilon}}

\newcommand\w{{\omega}}
\def\be  {\begin{equation}}
\def\ee  {\end{equation}}
\def\ba  {\begin{eqnarray}}
\def\ea  {\end{eqnarray}}
\def\baa {\begin{eqnarray*}}
\def\eaa {\end{eqnarray*}}
\newenvironment{comment}[2]
{\bgroup\vspace{7pt}
\begin{tabular}{|p{5in}|}
\hline \qquad \bf \footnotesize Comment -- to be deleted in the final version \\
\hline
\quad\sl\footnotesize #1#2} {\\ \hline \end{tabular}
\vspace{7pt}\indent\egroup}

\def\updots{\mathinner{\mkern
1mu\raise 1pt \hbox{.}\mkern 2mu \mkern 2mu \raise
4pt\hbox{.}\mkern 1mu \raise 7pt\vbox {\kern 7 pt\hbox{.}}} }

\newcommand{\B}{\mathbb B}

\newcommand{\R}{\mathbb R}

\newcommand{\N}{\mathbb N}

\renewcommand{\a}{\alpha}
\renewcommand{\b}{\beta}
\newcommand{\ineq}[1]{{\rm(\ref{#1})}}

\newcommand{\ie}{{\em i.e., }}
\newcommand{\eg}{{\em e.g. }}

\newcommand{\bpic}{
\begin{center}
}

\newcommand{\epic}{
\endpspicture
\end{center}
}

\newcommand{\st}{\;\; \big| \;\;}

\renewcommand{\L}{L}
\newcommand{\Lp}{\L_p}

\newcommand{\Poly}{\mathbb P}

\newcommand{\wkr}{\w_{k,r}^\varphi}
\newcommand{\wkrav}{\w_{k,r}^{*\varphi}}
 \newcommand{\Dom}{{\mathfrak{D}}}

\newcommand{\ddelta}{\mu(\delta)}

 \newcommand{\wt}{{\mathcal{W}}}

\newcommand{\ccc}{{\mathfrak  \vartheta}}

\newcommand{\thm}[1]{Theorem~\ref{#1}}
\newcommand{\lem}[1]{Lemma~\ref{#1}}
\newcommand{\cor}[1]{Corollary~\ref{#1}}

 \newcommand{\Lpab}{\L_p^{\alpha,\beta}}
 \newcommand{\wab}{w_{\a,\b}}
 \newcommand{\weight}{\wt_{kh}^{r/2+\a,r/2+\b}}

\title{{\sc On weighted approximation with Jacobi weights}
\thanks{{\it AMS classification:} 41A10, 41A17, 41A25. {\it Keywords
and phrases:} Approximation by polynomials in weighted $L_p$-norms, Degree
of approximation, direct and inverse theorems, Jacobi weights, moduli of smoothness, characterization of smoothness classes. }}

\author{K.  Kopotun\thanks{Department of Mathematics, University of
Manitoba, Winnipeg, Manitoba, R3T 2N2, Canada ({\tt
kirill.kopotun@umanitoba.ca}). Supported by NSERC of Canada.} ,
D. Leviatan\thanks{Raymond and Beverly Sackler School of Mathematical
Sciences, Tel Aviv University, Tel Aviv 69978, Israel ({\tt
leviatan@post.tau.ac.il}).}\ \ and I. A. Shevchuk\thanks
{Faculty of Mechanics and Mathematics, National Taras
Shevchenko University of Kyiv, 01033 Kyiv, Ukraine ({\tt
shevchuk@univ.kiev.ua}).} }

\begin{document}

\maketitle

\abstract{
We obtain matching direct and inverse theorems  for the degree of weighted $\Lp$-approximation by polynomials with the Jacobi weights $(1-x)^\a(1+x)^\b$. Combined, the estimates yield a constructive characterization of various smoothness classes of functions via the degree of their approximation by algebraic polynomials.
 }

\sect{Introduction and main results}

In this paper, we are interested in weighted polynomial approximation with the Jacobi weights
\[
w_{\a,\b}(x):=(1-x)^\a(1+x)^\b  , \quad \a,\b \in
J_p
 := \begin{cases}
(-1/p, \infty), & \mbox{\rm if } 0<p<\infty , \\
[0,\infty), & \mbox{\rm if } p=\infty .
\end{cases}
\]
Let
$\Lpab (I):=\left\{f\mid \norm{\wab f}{\Lp(I)}<\infty \right\}$, where $\norm{\cdot}{\Lp (I)}$ is the usual $\Lp$ (quasi)norm on the interval
$I\subseteq [-1,1]$, and, for
 $f\in \Lpab(I)$,
denote by
\[
E_n(f,I)_{\a,\b,p}:=\inf_{p_n\in\Pn} \norm{ \wab(f-p_n)}{\Lp(I)}
\]
the error of best weighted approximation of $f$ by polynomials in $\Pn$, the set of algebraic polynomials of degree not more than  $n-1$.
For $I=[-1,1]$, we denote $\norm{\cdot}{p}:= \norm{\cdot}{\Lp[-1,1]}$, $\Lpab := \Lpab ([-1,1])$, $E_n(f)_{\a,\b,p} := E_n(f,[-1,1])_{\a,\b,p}$, etc.

\begin{definition}[\cite{sam}] For $r\in\N_0$ and  $0< p \le \infty$, denote $\B^0_p(\wab) := \Lpab$ and
\[
\B_p^r(\wab):=\left\{\,f\,|\,f^{(r-1)}\in AC_{loc}(-1,1)\quad\text{and}\quad\varphi^rf^{(r)}\in \Lpab \right\}, \quad r\ge 1,
\]
where $\varphi(x):= \sqrt{1-x^2}$ and $AC_{loc}(-1,1)$ denotes the set of  functions which are locally absolutely continuous in $(-1,1)$.
\end{definition}

 As is common when dealing with $L_p$ spaces, we will not distinguish between a function in $\B_p^r(\wab)$ and all functions which are equivalent to it in $\Lpab$.

\begin{definition}[\cite{sam}] \label{maindefinition}
For $k,r\in\N$ and
$f\in \B^r_p(w_{\a,\b})$, $0< p\le\infty$,  define
\be \label{wkrdefinition}
\wkr(f^{(r)},t)_{\a,\b,p}:=\sup_{0\leq h\leq t}\norm{ \weight(\cdot)
\Delta_{h\varphi(\cdot)}^k (f^{(r)},\cdot)}p,
\ee
where
\[
\wt_\delta^{\xi,\zeta} (x):=  (1-x-\delta\varphi(x)/2)^\xi
(1+x-\delta\varphi(x)/2)^\zeta,
\]
and
\[
\Delta_h^k(f,x):=\left\{
\begin{array}{ll}
\sum_{i=0}^k  \binom{k}{i}
(-1)^{k-i} f(x-\frac{kh}{2}+ih),&\mbox{\rm if }\, [x-\frac{kh}{2}, x+\frac{kh}{2}]  \subseteq [-1,1] \,,\\
0,&\mbox{\rm otherwise},
\end{array}\right.
\]
is the $k$th symmetric difference. 
\end{definition}

For $\delta>0$, denote (see \cite{kls})
\begin{align*}
\Dom_\delta &:= \left\{x\st1-\delta\varphi(x)/2\geq|x|
\right\}\setminus\{\pm1\}
  =[-1+\ddelta,1-\ddelta],
\end{align*}
where
\[
\ddelta:=2\delta^2/(4+\delta^2).
\]
We note that $\Dom_{\delta_1}\subset\Dom_{\delta_2}$ if $\delta_2<\delta_1\le2$, and that $\Dom_\delta=\emptyset$ if $\delta>2$. Also,
since ${\Delta}_{h\varphi(x)}^k(f,x)=0$   if $x\not\in\Dom_{kh}$,
\be\label{dom}
\wkr(f^{(r)},t)_{\a,\b,p}=\sup_{0<h\leq t}\norm{\weight(\cdot)
\Delta_{h\varphi(\cdot)}^k(f^{(r)},\cdot)}{L_p(\Dom_{kh})} .
\ee
In particular,
$\wkr(f^{(r)},t)_{\a,\b,p}=\wkr(f^{(r)},2/k)_{\a,\b,p}$, for all $t\geq 2/k$.

Following \cite{sam} we also define the weighted averaged moduli.
\begin{definition}[\cite{sam}]
For $k\in\N$, $r\in\N_0$ and  $f\in \B^r_p(w_{\a,\b})$, $0< p<\infty$,    the $k$th weighted averaged modulus of smoothness of $f$ is defined as
\[
\wkrav(f^{(r)},t)_{\a,\b,p}
:=\left(\frac1t\int_0^t\int_{\Dom_{k\tau}}
| \wt^{r/2+\a,r/2+\b}_{k\tau}(x)\Delta^k_{\tau\varphi(x)}(f^{(r)},x)|^p\,dx\,d\tau
\right)^{1/p} .
\]
If $p=\infty$ and   $f\in\B^r_\infty(w_{\a,\b})$, we write
\be\label{ineq1}
\wkrav(f^{(r)},t)_{\a,\b,\infty}:=\wkr(f^{(r)},t)_{\a,\b,\infty}\,.
\ee
\end{definition}
Clearly, 
\be\label{ineq}
\omega_{k,r}^{*\varphi}(f^{(r)},t)_{\a,\b,p}\le\wkr(f^{(r)},t)_{\a,\b,p} ,\quad t>0.
\ee

Moreover, it was proved in \cite{sam} that if $r/2+\a,r/2+\b\ge0$, then the weighted moduli and the weighted averaged moduli are equivalent.

\medskip

Throughout this paper, all constants $c$ may depend only on $k$, $r$, $p$, $\a$ and $\b$, unless a specific dependence on an additional parameter is mentioned.

\medskip

We have the following direct (Jackson-type) theorem.

\begin{theorem}\label{direct}
Let $k\in\N$,   $0< p\le\infty$, $ \a\geq 0$ and $\b\geq 0$. If $f\in\Lpab$, then
\be\label{dir}
E_n(f)_{\a,\b,p}\le c \w_{k,0}^\varphi (f,1/n)_{\a,\b,p}\,,\quad n\ge k.
\ee
\end{theorem}

It follows from \cite{sam}*{Lemma 1.11} that, if $k\in\N$, $r\in\N_0$,   $r/2+\a\geq 0$, $r/2+\b\ge 0$,   $1\le p\le\infty$ and $f\in\B^{r+1}_p(\wab)$, then
\[
\w_{k+1,r}^\varphi (f^{(r)},t)_{\a,\b,p}\le c t \w_{k,r+1}^\varphi(f^{(r+1)},t)_{\a,\b,p}, \quad t>0.
\]
Hence, \ineq{dir} implies that, for $f\in B_p^r(\wab)$, $1\le p \leq \infty$,
\[
E_n(f)_{\a,\b,p}\le c \w_{k+r,0}^\varphi (f,1/n)_{\a,\b,p} \leq c t^r  \wkr (f^{(r)},1/n)_{\a,\b,p}   \,,\quad n\ge k+r,
\]
provided $\a,\b \geq 0$. We    strengthen this result by showing that the last estimate  is, in fact, valid for all $\a,\b \geq -r/2$. Namely,

\begin{theorem}\label{thm2direct}
Let $k\in\N$, $r\in\N_0$, $1\le p\le\infty$, and $\a,\b\in J_p$ be such that $r/2+\a\geq 0$ and $r/2+\b\geq 0$. If $f\in\B^r_p(w_{\a,\b})$, then
\be\label{thm2dir}
E_n(f)_{w_{\a,\b},p}\le\frac c{n^r}\wkr(f^{(r)},1/n)_{\a,\b,p}\,,\quad n\ge k+r.
\ee
\end{theorem}

We remark that \thm{thm2direct} is not valid if $r\geq 1$ and $0<p<1$ (one can show this using the same construction that was used in the proof of \cite{k-at8}*{Theorem 3 and Corollary 4}).

Next, we have the following inverse result in the case $1\le p \le \infty$.
\begin{theorem} \label{inverse}
Suppose that $r\in\N_0$, $1\le p\le \infty$, $\a,\b\in J_p$ are such that $r/2+\a\geq 0$ and $r/2+\b\geq 0$, and $f\in \Lpab$.
If
\be\label{rcond}
\sum_{n=1}^\infty rn^{r-1}E_n(f)_{\wab,p}<+\infty
\ee
$($\ie if $r=0$ then this condition is vacuous$)$,
then 
$f\in\B^r_p(w_{\a,\b})$, and for $k\in\N$ and $N\in\N$,
\begin{eqnarray}\label{inverse1}
\wkr( f^{(r)},t)_{\a,\b,p}&\le&
c \sum_{n>\max\{N,1/t\}}rn^{r-1}E_n(f)_{w_{\a,\b},p}\\ \nonumber
&&\quad+ct^{k}\sum_{N\leq n\leq\max\{N,1/t\}}n^{k+r-1}E_n(f)_{w_{\a,\b},p}
\\&&\quad+c(N)t^{k}E_{k+r}(f)_{w_{\a,\b},p}\,,\quad t>0.\nonumber
\end{eqnarray}
In particular, if $N\le k+r$, then
\begin{eqnarray*} 
\wkr( f^{(r)},t)_{\a,\b,p}&\le&
c\sum_{n >\max\{N,1/t\}}rn^{r-1}E_n(f)_{w_{\a,\b},p}\\ \nonumber
&& +ct^{k}\sum_{N\leq n\leq\max\{N,1/t \}}n^{k+r-1}E_n(f)_{w_{\a,\b},p} , \; t>0.
\end{eqnarray*}
\end{theorem}

\begin{remark}\label{rem}{\rm(i)} Note that the first term in \ineq{inverse1} disappears if $r=0$.

\noindent{\rm(ii)} If $\a=\b=0$, Theorem $\ref{inverse}$ was proved in \cite{kls}.

\noindent{\rm(iii)} The case $\a,\b\ge0$, $N=1$ and $r=0$ follows from \cite{dt}*{Theorem 8.2.4} by virtue of \ineq{claim}.
\end{remark}

Denote by $\Phi$  the set of nondecreasing functions
$\phi:[0,\infty)\rightarrow[0,\infty)$, satisfying $\lim_{t\to 0^+}\phi(t)=0$. The following is an immediate corollary of \thm{inverse} (in fact, it is a restatement of \thm{inverse} in terms of $\phi$).

\begin{corollary} \label{cor212}
Suppose that $r\in\N_0$, $N\in\N$, $1\le p\le \infty$, $\a,\b\in J_p$ are such that $r/2+\a\geq 0$ and $r/2+\b\geq 0$, and $\phi\in\Phi$ is such that
\[
 \int_0^1 \frac{r \phi (u)}{u^{r +1}}du<+\infty
\]
$($\ie if $r=0$ then this condition is vacuous$)$.
Then, if $f\in \Lpab$ is such that
\[
E_n(f)_{\wab,p} \le\phi\left(\frac1{n+1}\right),\quad \mbox{for all} \quad n\ge N,
\]
then
$f\in\B^r_p(w_{\a,\b})$, and for $k\in\N$ and $0<t\leq 1/2$,
\[
\wkr( f^{(r)},t)_{\a,\b,p} \le
c\int_0^t\frac{r\phi(u)}{u^{r+1}}du
+ ct^{k}\int_t^1\frac{\phi(u)}{u^{k+r+1}}du
 +c(N)t^{k}E_{k+r}(f)_{w_{\a,\b},p}. \nonumber
\]
In particular, if $N\le k+r$, then
\[
\wkr( f^{(r)},t)_{\a,\b,p} \le
c\int_0^t\frac{r\phi(u)}{u^{r+1}}du
+ ct^{k}\int_t^1\frac{\phi(u)}{u^{k+r+1}}du.
\]

\end{corollary}

\begin{remark}
We take this opportunity to correct an inadvertent misprint in three of our earlier papers where the inverse theorems of this type were proved in the case $\a=\b=0$. Namely, the inequality $E_n(f)_{p} \le\phi\left(1/n\right)$ in \cite{kls-umzh}*{Theorem 3.2} $($the case $p=\infty)$, and in
\cite{kls}*{Theorem 9.1} and \cite{kls1}*{Theorem $I_r$} $($the case $1\le p\le\infty)$,
should be replaced by $E_n(f)_{p} \le\phi\left(1/(n+1)\right)$. Otherwise, the last estimates in these results are not justified/valid if $N=1$ , $k=1$ and $r=0$ since $E_{k+r}(f)_p = E_1(f)_p \le \phi\left(1\right)$ cannot be estimated above by $\int_t^1 \phi(u) u^{-2} du$ without any extra assumptions on the function $\phi$.
\end{remark}

It immediately follows from   \thm{direct} that if $\a,\b\in J_p$, $r/2+\a\geq 0$, $r/2+\b\geq 0$ and
 $\wkr(f^{(r)},t)_{\a,\b,p}\le t^\gamma$,   then $E_n(f)_{w_{\a,\b},p}\le cn^{-r-\gamma}$.
Conversely, an immediate consequence of \thm{inverse} (\cor{cor212}) is the following result which, for $\a,\b\ge0$, was proved by a different method in \cite{kls1}*{Theorem 5.3}.

\begin{corollary}
Suppose that $r\in\N_0$, $N\in\N$, $1\le p\le \infty$, and $\a,\b\in J_p$ are such that $r/2+\a\geq 0$ and $r/2+\b\geq 0$.
If $f\in \Lpab$ is such that, for some $N\in\N$ and $r< \gamma <k+r$,
\be\label{Aleph}
E_n(f)_{\wab,p}\le n^{-\gamma},\quad n\ge N,
\ee
then
$f\in\B^r_p(\wab)$,  and
\[
\wkr(f^{(r)},t)_{\a,\b,p}\le
c t^{\gamma-r}+c(N)t^k E_{k+r}(f)_{w_{\a,\b},p},\quad t>0.
\]
In particular, if $N\le k+r$, then
\[
\wkr(f^{(r)},t)_{\a,\b,p}\le c t^{\gamma-r},\quad t>0.
\]
\end{corollary}

Finally, we have the following inverse theorem for $0<p<1$ which is an immediate corollary of \cite{k-sing}*{Theorem 10.1} and \cite{kls-umzh}*{Lemma 4.5}.



\begin{theorem}
Let $k \in\N$,  $\a\geq 0$, $\b\geq 0$  and $f\in\Lpab$, $0<p <1$. Then there exists a positive constant $\ccc \leq 1$ depending only on   $k$, $p$, $\a$ and $\b$ such that, for any $n\in\N$,
\[
\w_{k,0}^{\varphi}(f, \ccc n^{-1})_{\a,\b,p}^p \le
 c n^{-kp}   \sum_{m=1}^{n}    m^{kp-1  }     E_{m}(f)_{\wab, p}^p .
\]
\end{theorem}

\sect{Auxiliary lemmas}

\begin{lemma}\label{der}
Let $k\in\N$,     $0< \delta  \le 2$, and let $y:=y(x)$,
  $y:[-1,1] \to \R$ be   such that
\[
y(x)+\delta\varphi(y(x))/2 =x, \quad x\in [-1,1].
\]
Then,
\begin{enumerate}[(i)]
\item \label{i}
$y=y(x)$ is strictly increasing on $[-1,1]$, and  $y'(x)\le 2$, $x\in [-1,1]$,
\item \label{ii}
$y\left([-1+2\ddelta,1]\right) =  \Dom_{\delta}$,
\item \label{iii}
$y'(x)\ge 2/3$, $x\in [-1+2\ddelta,1]$,
\item  \label{iv}
if $y_\lambda(x):=y(x)+ \lambda \varphi(y(x))$, then
$1/3\le y_\lambda'(x)\le 3$, for all $|\lambda| \le \delta/2$ and $x\in [-1+2\ddelta,1]$,
\item \label{v}
for all $x \in [-1+2\ddelta,1]$,
 \be \label{v1}
\ddelta + 2(1-x)/3  \le 1-y(x) \le \ddelta + 2(1-x)
\ee
and
 \be \label{v2}
(1+x)/2 \le 1+y(x) \le 1+x.
\ee

\end{enumerate}
\end{lemma}

\begin{proof} Since $x\le 1$, we have $y+\delta \varphi(y)/2\le1$ which can be rewritten as
%
$ \delta/(2\varphi(y)) \le  1/(1+y)$,
and so, if $y\ge 0$, then
\[
1 -  \frac{\delta y}{2 \varphi(y)} \geq \frac1{1+y } \geq \frac 12,
\]
and, clearly, $1 -   \delta y /(2 \varphi(y)) \ge  1/2$ if $y<0$ as well.

Therefore, since
\[
\frac{dy}{dx} = \left( 1 -  \frac{\delta y}{ 2 \varphi(y)} \right)^{-1},
\]
we immediately conclude that \ineq{i} holds.

Now, since $y$ is nondecreasing,
$y\left([-1+2\ddelta,1]\right) = \left[ y(-1+2\ddelta), y(1)\right]$, and \ineq{ii} follows because
$y(1)=1-\ddelta$ and $y(-1+2\ddelta)=-1+\ddelta$.

It follows from \ineq{ii} that, for $x\in [-1+2\ddelta,1]$, we have $y-\delta \varphi(y)/2\ge -1$, which can be rewritten as
$ \delta/(2\varphi(y)) \le  1/(1-y)$, and so, if $y\le 0$, then
\[
1 -  \frac{\delta y}{ 2 \varphi(y)} \le \frac{1-2y}{1-y} \le \frac 32,
\]
and, clearly, $1 -   \delta y /(2 \varphi(y)) \le  3/2$ if $y> 0$ as well. This implies \ineq{iii}.

Now, it follows from the above estimates that $ \delta/(2\varphi(y)) \le  1/(1+|y|)$,  for $x\in [-1+2\ddelta,1]$, which implies
\[
y_\lambda'(x) =\left(1 - \frac{\lambda  y }{\varphi(y)}\right) y'(x) \le
2 \left(1 +\frac{\delta  |y| }{2\varphi(y)}\right)
\le
\frac{2+4|y|}{1+|y|} \le 3 ,
\]
and
\[
y_\lambda'(x)  \ge
\frac 23 \left(1 - \frac{\delta  |y| }{2\varphi(y)}\right) \ge
    \frac{2}{3(1+|y|)}
    \ge \frac 13 ,
\]
and so \ineq{iv} is verified.

Now, by
 \[
 \frac{dy}{dx}(\xi)=\frac{y(1)- y(x)}{1-x} =  \frac{1-\ddelta- y(x)}{1-x}
 \]
 \ineq{i} and \ineq{iii} imply, for $x\in [-1+2\ddelta,1]$,
 \[
\ddelta + 2(1-x)/3  \le 1-y(x) \le \ddelta + 2(1-x)  ,
\]
which is \ineq{v1}. Finally, the second inequality in \ineq{v2} is obvious, and the first one immediately follows from \ineq{ii} which implies
\[
1+x = 1+y + \delta \varphi(y)/2 \le 2(1+y).
\]
Thus, \ineq{v} is verified.
\end{proof}

\sect{Whitney-type estimates}

In this section, we prove Whitney-type estimates, which we feel are of independent interest, and which we need in order to prove the direct (Jackson-type) theorem (\thm{direct}) for small $n$.

Recall that the celebrated Whitney inequalities for the ordinary moduli of smoothness were first proved by Whitney \cite{w}  for functions in $C[a,b]$. Later Brudnyi \cite{b} extended the inequalities to $L_p[a,b]$, $1\le p<\infty$ and, finally, Storozhenko \cite{st} proved the inequalities for $L_p[a,b]$, $0<p<1$.

\begin{theorem} \label{thm:localwh}
Let $k\in\N$, $\a\geq 0$, $\b\geq 0$,  $0< p\leq \infty$,  $f\in \Lpab$, $0<h\le 2$ and $x_0\in\Dom_{h}$. Then, for any $\theta\in(0,1]$, we have
\[ 
E_k(f, [x_0- h\varphi(x_0)/2, x_0+h\varphi(x_0)/2])_{\wab, p}  \le c \w_{k,0}^{*\varphi}(f,\theta h)_{\a,\b,p}  \le c \w_{k,0}^\varphi(f,\theta h)_{\a,\b,p} ,
\]
where $c$ depends only on $k$,   $\a$, $\b$, $p$ and $\theta$.
\end{theorem}

Choosing $x_0=0$ and $h=2$   in \thm{thm:localwh} (and replacing $2\theta$ by $\theta$) we immediately get the following corollary.

\begin{corollary} \label{whcor}
Let $k\in\N$, $\a\geq 0$, $\b\geq 0$, $0< p\leq \infty$ and $f\in \Lpab$. Then, for any $\theta\in(0,1]$, we have
\be \label{ineq:regwh}
E_k(f)_{\wab,p} \le c \w_{k,0}^{*\varphi}(f,  \theta)_{\a,\b,p} \le c \w_{k,0}^\varphi(f,  \theta)_{\a,\b,p} ,
\ee
where $c$ depends only on $k$,   $\a$, $\b$, $p$ and $\theta$.
\end{corollary}
Also, if $x_0\pm h\varphi(x_0)/2=\pm1$,   \thm{thm:localwh} immediately gives the following result (by letting
$h:= t \sqrt{4A/(4-At^2)}$, $x_0 := \pm(1- \mu(h))$,
  $\theta := \min\{1, 1/\sqrt{2A}\}$, and using monotonicity of the moduli with respect to $t$).

 \begin{corollary} \label{corwh}
Let $k\in\N$, $\a\geq 0$, $\b\geq 0$, $A>0$, $0< p\leq \infty$ and $f\in \Lpab$. Then, for any $0<t\leq \sqrt{2/A}$, we have
\[
E_k(f, [1-At^2,1])_{\wab, p} \le c \w_{k,0}^{*\varphi}(f,t)_{\a,\b,p} \le c \w_{k,0}^\varphi(f,t)_{\a,\b,p} ,
\]
and
\[
E_k(f, [-1,-1+At^2])_{\wab, p} \le c \w_{k,0}^{*\varphi}(f,t)_{\a,\b,p} \le c \w_{k,0}^\varphi(f,t)_{\a,\b,p} ,
\]
where $c$ depends on $k$, $p$, $\a$, $\b$ and $A$.
\end{corollary}

\begin{proof}[Proof of $\thm{thm:localwh}$]  \thm{thm:localwh}  follows from   the classical (non-weighted) Whitney's inequality (see \cite{dl}*{Theorem 6.4.2 and Theorem 12.5.5}), which readily implies (see \eg \cite{pp}*{Sections 3.1 and 7.1}), for each interval $J\subset[-1,1]$,
the existence of a polynomial $p_k\in\Poly_k$, such that
\be \label{initial}
\norm{f-p_k}{\Lp(J)} \leq c \w_k(f,  |J|; J)_p\le  c \frac{|J|^{k-1+1/p_1}}{\delta^{k-1+1/p_1}}\w_k\left(f,  \delta; J\right)_p,\quad 0<\delta\le |J|,
\ee
where $|J|$ is the length of of the interval $J$ and $p_1:=\min\{1,p\}$.

In order to prove  \thm{thm:localwh}, we assume,
without loss of generality,  that $x_0\ge 0$, and denote
$$
[a,b]:=[x_0-h\varphi(x_0)/2, x_0+h\varphi(x_0)/2],\qquad
W_p:=\w_{k,0}^{*\varphi}(f, \theta h)_{\a,\b,p},
$$
Note that
\be\label{11}
1-x\le2(1-x_0)\quad\text{and}\quad 1+x\le2(1+x_0),\quad x\in[a,b],
\ee
since $x_0$ is the middle of $[a,b]$, and so
\be \label{123}
\varphi(b) \le \varphi(x) \le 2 \varphi(x_0) , \quad \text{for all} \quad   x\in [a,b] ,
\ee
where the first inequality is valid since $|x|\le |b|$ (because $x_0$ is assumed to be nonnegative).

We will consider two cases: (i) $\varphi(x_0)\le2\varphi(b)$ and (ii)   $\varphi(x_0) > 2\varphi(b)$.

\medskip
\noindent
{\bf Case (i):}
$\varphi(x_0)\le2\varphi(b)$, $x\in [a,b]$.

Then, for all $x\in[a,b]$,
\be \label{44}
 1-x_0\le \varphi^2(x_0)\le 4\varphi^2(b)\le 4\varphi^2(x)<8(1-x)
\ee
and
\be\label{44a}
  1+x_0= \frac{\varphi^2(x_0)}{1-x_0}\le\frac{ 4\varphi^2(b)}{1-x_0}\le\frac{ 8\varphi^2(x)}{1-x}=8(1+x).
\ee

Now, let $J:=[a,b]$ and $\delta:=\theta h\varphi(b)$, and note that
\be\label{88}
\frac\theta2|J| = \frac\theta2 h\varphi(x_0) \le  \delta \le \theta h\varphi(x_0) \le |J|.
\ee
So, for $p=\infty$, we have
\begin{align*} \label{initial}
 \w_k(f,\delta; J)_\infty&= \sup_{0<s\le\delta} \sup_{x\in J} \left|  \Delta_{s}^k (f,x; J)\right|
= \sup_{0<\tau \leq \delta/\varphi(b)} \sup_{x\in J} \left|  \Delta_{\tau\varphi(b)}^k (f,x; J)\right|\\
&= \sup_{0<\tau\leq \theta h} \sup_{x\in J} \left|  \Delta_{\tau\varphi(b)}^k (f,x; J)\right|
\le\sup_{0<\tau\leq \theta h} \sup_{x\in J} \left|  \Delta_{\tau\varphi(x)}^k (f,x; J)\right|\\
&\le c\wab^{-1}(x_0)\sup_{0<\tau\leq \theta h} \sup_{x\in J} \left| \wt_{k\tau}^{\a,\b}(x) \Delta_{\tau\varphi(x)}^k (f,x; J)\right|\\
&=c\wab^{-1}(x_0)W_{\infty},
\end{align*}
where in the last inequality we used the fact that the estimates \ineq{44} and \ineq{44a} imply that
$\wab(x) \le c \wt_{k\tau}^{\a,\b}(x)$, for all $x$ such that $x\pm k\tau\varphi(x)/2 \in J$.

 If $p<\infty$, then it is well known (see \eg \cite{pp}*{Lemma 7.2}) that
 \[
 \w_k(f, t; J)_p^p \leq c \frac{1}{t} \int_0^t \int_J |\Delta_{s}^k (f,x;J)|^p dx ds,\quad 0< t \le |J|/k.
 \]
 Hence, using \ineq{88} and \ineq{123} we have
 \begin{align*}
 c\delta\w_k(f, \delta; J)_p^p & \le      \int_J \int_0^{\delta} |\Delta_{s}^k (f,x,J)|^p ds dx \\
 & = \int_J \int_0^{\delta/\varphi(x)}  \varphi(x) |\Delta_{\tau \varphi(x)}^k (f,x,J)|^p d\tau dx \\
& \le
  \int_J \int_0^{\theta h}  \varphi(x) |\Delta_{\tau \varphi(x)}^k (f,x,J)|^p d\tau dx \\
  & \le
  c \wab^{-p}(x_0) \varphi(b)   \int_J \int_0^{\theta h}  | \wt_{k\tau}^{\a,\b}(x) \Delta_{\tau \varphi(x)}^k (f,x,J)|^p d\tau dx \\
  & \leq
 c  \wab^{-p}(x_0) \varphi(b)   \int_0^{\theta h} \int_{\Dom_{k\tau}} | \wt_{k\tau}^{\a,\b}(x) \Delta_{\tau \varphi(x)}^k (f,x)|^p dx  d\tau \\
 & =
 c  \wab^{-p}(x_0)  \theta h \varphi(b)  W_p^p.
 \end{align*}
Thus, for all $0< p\le\infty$, we have
$$
\w_k(f, \delta; J)_p\le c\wab^{-1}(x_0)W_p ,
$$
which, by virtue of  \ineq{11}, yields
\[
\norm{\wab(f-p_k)}{\Lp(J)} \leq c\wab(x_0) \norm{f-p_k}{\Lp(J)} \leq
c\wab(x_0) \w_k\left(f,  \delta; J\right)_p\le cW_p,
\]
and so the proof is complete in Case (i).

\medskip
\noindent
{\bf Case (ii):}  $\varphi(x_0)>2\varphi(b)$.

We first note that, in this case, it suffices to assume that $b=1$. Indeed, suppose that  the theorem is proved for all $\hat x_0$ and $\hat h$ such that  $\hat x_0+\hat h\varphi(\hat x_0)/2=1$, and let $x_0$ and $h$ be such that $\varphi(x_0)>2\varphi(b)$ (recall that $b =   x_0+  h\varphi( x_0)/2$).
We let $\hat x_0 := x_0$, $\hat h := 2(1-x_0)/\varphi(x_0)$ and note that $x_0+\hat h\varphi(x_0)/2=1$.
Now, since
\[
1-x_0 = \frac{\varphi^2(x_0)}{1+x_0} > \frac{4\varphi^2(b)}{1+x_0} = 4(1-b) \frac{1+b}{1+x_0} \ge 4 (1-b) ,
\]
we have
\[
h\varphi(x_0) = 2(b-x_0) = 2(1-x_0)-2(1-b) > 3(1-x_0)/2 .
\]
Therefore,
$h\le \hat h \le 4h/3$, and so
\begin{align*}
\lefteqn{ E_k(f, [x_0- h\varphi(x_0)/2, x_0+h\varphi(x_0)/2])_{\wab, p} }\\
&\quad \le E_k(f, [x_0- \hat h\varphi(x_0)/2, x_0+\hat h\varphi(x_0)/2])_{\wab, p} \\
&\quad \le
c \w_{k,0}^{*\varphi}(f,\theta_1 \hat h)_{\a,\b,p}
\le
cW_p ,
\end{align*}
where $\theta_1 := 3\theta/4$.

\medskip

Hence, for the rest of this proof, we assume that $b=1$. Note that
\be \label{vari}
b-a =  h\varphi(x_0) = 2(1-x_0) = 2\mu(h) \in [h^2/2, h^2] .
\ee

Define
\[
\tilde h:=\frac{\theta h}{10k}, \quad \tilde b:=1-\tilde h^2\quad\text{and}\quad  J:=[a,b]\cap[-\tilde b,\tilde b].
\]
Then $x_0\in J$, and, for all $x\in J$, we have
\be\label{99}
\frac{1-x_0}{1-x}\le\frac{\mu(h)}{\tilde h^2}<c,\quad\frac{1+x_0}{1+x} \le \frac{2}{\max\{\tilde h^2, 1+a\}} \le \frac{c}{\max\{ h^2, 4-h^2\}}  <c,
\ee
 and
\[
\varphi(\tilde b)\le \varphi(x)\le 2\varphi(x_0) \le c\varphi(\tilde b) .
\]
We now let $\delta:= \theta h\varphi(\tilde b)$, note that
\[
c|J| \le c(b-a) \le  \delta \le b-a \le c|J| ,
\]
and conclude using the same argument that was used in Case (i) that
 there is a polynomial
$p_k\in\Poly_k$, such that
\be\label{66}
\norm{\wab(f-p_k)}{\Lp( J)} \leq   cW_p.
\ee

So, to finish the proof in   Case (ii) we have to show that, for the function
$g:=f-p_k$,
the inequalities
\be\label{55}
\norm{\wab g}{\Lp[\tilde b,1]} \leq c   W_p.
\ee
and, if $a<-\tilde b$,
\be\label{56}
\norm{\wab g}{\Lp[a,-\tilde b]} \leq c   W_p.
\ee
hold. We prove \ineq{55}, the proof of \ineq{56} being similar.

 To this end let $t\in [ 2\tilde h/\sqrt{k}, 4\tilde h/\sqrt{k}]$ be fixed for now, and denote
  by $y=y(x)$ and $y_i=y_i(x)$, $1\leq i \leq k$, the functions  such that 
\[
y(x)+kt\varphi(y(x))/2=x \andd y_i(x):=x-it\varphi(y(x)) = y(x) +(k/2-i) t \varphi(y(x)) .
\]

%
%
%

We now note that $[\tilde b, 1]\subset [-1+2\mu(kt),1]$,
since
\[
-1+2\mu(kt) \le -1+k^2t^2 \le -1+16k \tilde h^2 \le 1-\tilde h^2 = \tilde b
\]
and so \lem{der} with $\delta = kt$ implies that, for all $x\in [\tilde b, 1]$, $2/3\le y'(x)\le 2$,  $1/3\le y_i'(x)\le 3$, and
\be \label{var}
\varphi^2(y(x)) \le 2(\mu(kt)+2 \tilde h^2) \le k^2t^2 + 4\tilde h^2   \le 25k\tilde h^2 .
\ee
Additionally, note that
\be \label{777}
y_i(x)\in J , \quad x\in [\tilde b, 1] \andd 1\le i\le k .
\ee
 Indeed,  since $y(1)=1-\mu(kt)$, we have, for $x\in [\tilde b, 1]$,
\[
y_i(x) \le y_1(x) \le y_1(1) = 1-t \varphi(y(1))= 1- 2\mu(kt)/k \le 1- kt^2/2 \le 1-2\tilde h^2 < \tilde b ,
\]
 and, using \ineq{var} and \ineq{vari},
\begin{align*}
y_i(x) & \ge y_k(x) \ge y_k(\tilde b) = \tilde b - kt\varphi(y(\tilde b)) \ge 1-\tilde h^2 - 5  k^{3/2} t \tilde h
\ge 1-21k \tilde h^2\\
& \ge \max\{-1+\tilde h^2, a\},
\end{align*}
which yields \ineq{777}. Note also that the above implies that $1+y (x) \ge 3kt\varphi(y(x))/2$, for  $x\in [\tilde b, 1]$.

%

Hence,
\[
\wab(x) = \wab\left(y(x)+ kt  \varphi(y(x))/2 \right)\le 2^\b\wab(y_i(x)) , \quad x\in[\tilde b,1] ,
\]
and using \ineq{777} and \ineq{66} we get, for $0<p<\infty$,
\begin{align*}
\norm{\wab  g(y_i)}{\Lp[\tilde b,1]} & \le 2^\b \norm{\wab(y_i) g(y_i)}{\Lp[\tilde b,1]} \le c \norm{ \wab(y_i) g(y_i)(y_i')^{1/p}}{\Lp[\tilde b,1]}  \\
&\le   c \norm{ \wab  g }{\Lp(J)}     \le
 cW_p .
\end{align*}
If $p=\infty$, then  similar (and, in fact, simpler) arguments yield
\[
  \|\wab g(y_i)\|_{L_\infty[\tilde b,1]} \leq  cW_\infty , \quad 1\le i \le k .
\]
Now, for $x\in [\tilde b, 1]$,
\begin{align*}
g(x) &=
 \Delta_{t\varphi(y(x))}^k (g,y(x)) - \sum_{i=0}^{k-1}(-1)^{k-i} \binom{k}{i}  g\left(y(x)+(i-\frac k2)t \varphi(y(x))\right)\\
&= \Delta_{t\varphi(y(x))}^k (g,y(x)) - \sum_{i=1}^{k}(-1)^{i} \binom{k}{i}    g\left(y_i(x)\right),
\end{align*}
and so
\begin{align*}
\|\wab g\|_{L_p[\tilde b,1]} & \leq
c \norm{\wab \Delta_{t\varphi(y)}^k (g,y) }{\Lp[\tilde b,1]}
+
c \sum_{i=1}^{k} \binom{k}{i}    \norm{\wab g(y_i)}{\Lp[\tilde b,1]}   \\
&\le c\norm{\wab \Delta_{t\varphi(y)}^k (g,y) }{\Lp[\tilde b,1]} + cW_p \\
&\le c\norm{\wt_{tk}^{\a,\b}(y) \Delta_{t\varphi(y)}^k (g,y) }{\Lp[\tilde b,1]} + cW_p \\
&\le c\norm{\wt_{tk}^{\a,\b} \Delta_{t\varphi}^k (g,\cdot) }{\Lp(\Dom_{kt})} + cW_p.
\end{align*}
This completes the proof in the case $p=\infty$. If $p<\infty$, then integrating with respect to $t$ over
  $[ 2\tilde h/\sqrt{k}, 4\tilde h/\sqrt{k}]$  we get
\[
\norm{\wab g}{\Lp[\tilde b,1]}^p \le \frac{c}{\tilde h}\int_{ 2\tilde h/\sqrt{k}}^{4\tilde h/\sqrt{k}}\norm{\wt_{t k }^{\a,\b} \Delta_{t\varphi}^k (g,\cdot) }{L_p(\Dom_{kt})}^p dt +cW_p^p\le cW_p^p.
\]
The proof is now complete.
\end{proof}

We now prove a Whitney-type result for functions from $f\in\B^r_p(\wab)$, $r\in\N$.

\begin{theorem} \label{whitneythm}
Let $k\in\N$, $r\in\N$, $1\le p\le\infty$,  and let $\a,\b\in J_p$ be such that  $r/2+\a\geq 0$ and $r/2+\b\geq 0$. If $f\in\B^r_p(\wab)$, then for any $\theta\in(0,1]$,
\be \label{whitneyineq}
E_{k+r}(f)_{\wab,p} \le c \wkr(f^{(r)}, \theta)_{\a,\b,p} .
\ee
\end{theorem}

\begin{proof}
Note that $f\in\B_p^r(\wab)$ implies that $f^{(r)}\in L_p^{r/2+\a, r/2+\b}$, and so it follows from \ineq{ineq:regwh}  that
\[
E_k(f^{(r)})_{w_{r/2+\a,r/2+\b,p}} \le c\omega_{k,0}^\varphi(f^{(r)}, \theta)_{r/2+\a,r/2+\b,p}=c W_{r,p},
\]
where $W_{r,p}:= \wkr(f^{(r)},  \theta)_{\a,\b,p}$.

Let
$\tilde P_k\in\Poly_k$ be a polynomial  such that
\[
\norm{w_{r/2+\a,r/2+\b}(f^{(r)}-\tilde P_k)}{p} <c W_{r,p},
\]
and define $P_{k+r}\in \Poly_{k+r}$ by
\[
P_{k+r}(x) :=f(0)+\frac{f'(0)}{1!}x+\dots+ \frac{f^{(r-1)}(0)}{(r-1)!}x^{r-1}+\frac1{(r-1)!}\int_0^x(x-u)^{r-1}\tilde P_k(u)du.
\]
Assuming that $x\ge 0$ (for $x<0$ all estimates are analogous), we have by H\"older's inequality
\begin{eqnarray*}
\lefteqn{(r-1)!\left|f(x)-P_{k+r}(x)\right|}\\
&\le&
\int_0^x(x-u)^{r-1}\left|f^{(r)}(u)-\tilde P_k(u)\right|du\\
&=&
\int_0^x\frac{(x-u)^{r-1}}{w_{r/2+\a,r/2+\b}(u)}w_{r/2+\a,r/2+\b}(u)|f^{(r)}(u)-\tilde P_k(u))|du\\
&\le&
 c A_q(x) W_{r,p},
\end{eqnarray*}
where $q:=p/(p-1)$,
\[
A_q(x):= \left(\int_0^x\left(\frac{(x-u)^{r-1}}{w_{r/2+\a,r/2+\b}(u)}\right)^q du\right)^{1/q}, \quad \text{if} \quad q<\infty,
\]
and
\[
A_\infty(x):= \sup_{u\in \left[0, x \right]}  \left(\frac{(x-u)^{r-1}}{w_{r/2+\a,r/2+\b}(u)}\right) .
\]
Now, since
\[
\frac{(x-u)^{r-1}}{w_{r/2+\a,r/2+\b}(u)}\le \frac{(x-u)^{r-1}}{(1-u)^{r/2+\a}}\le  (1-u)^{r/2-\a-1},
\]
we have
\[
A_q^q(x)\le  \int_0^x(1-u)^{q(r/2-\a-1)}du \andd A_\infty(x) \le  \sup_{u\in \left[0, x \right]} (1-u)^{r/2-\a-1}.
\]
If $q<\infty$ and $q(r/2-\a-1)>-1$, then
\[
A_q^q(x)\le \int_0^1(1-u)^{q(r/2-\a-1)}du=c,
\]
which yields
$$
\|f-P_{k+r}\|_{L_\infty[0,1]}\le c  W_{r,p},
$$
and hence
\begin{align} \label{123321}
\|w_{\a,\b}(f-P_{k+r})\|_{L_p[0,1]}&\le c  W_{r,p} \|w_{\a,\b}\|_{L_p[0,1]}
 \le cW_{r,p},
\end{align}
where we used the fact that $\a \in J_p$. Similarly, \ineq{123321} holds if $q=\infty$ ($p=1$) and $r/2-\a-1 \geq 0$.

 If $q<\infty$ and $q(r/2-\a-1)<-1$, then
\[
A_q^q(x) \le c(1-x)^{q(r/2-\a-1)+1},
\]
and so, recalling that $x\ge 0$,  we have
\[
\wab (x) A_q(x)\le c(1-x)^{r/2-1/p}.
\]
Hence,
\be \label{135}
\|w_{\a,\b}(f-P_{k+r})\|_{L_p[0,1]}\le c \|w_{\a,\b}A_q\|_{L_p[0,1]} W_{r,p} \le c W_{r,p}.
\ee
 Similarly, one shows that \ineq{135} holds if
$q=\infty$ ($p=1$) and $r/2-\a-1 < 0$.

 It remains to consider the case $q<\infty$ and $q(r/2-\a-1)=-1$. We have
\[
A_q^q(x)\le  \int_0^x(1-u)^{-1}du= c|\ln(1-x)|,
\]
and so
\[
w_{\a,\b}(x) A_q(x)\le c(1-x)^{\a}|\ln(1-x)|^{1/q} .
\]
For $p<\infty$, since $\a p> -1$, we have
\[
\norm{\wab A_q}{L_p[0,1]}^p \le  c  \int_0^1 (1-x)^{\a p}|\ln(1-x)|^{p-1}  dx < c.
\]
Finally, if $p=\infty$, then $q=1$ and $\a = r/2>0$, and so $\norm{\wab A_1}{L_\infty[0,1]} < c$.
Hence, \ineq{135} holds in this case as well.

Similarly, one shows  that
\[
\|\wab(f-P_{k+r})\|_{L_p[-1,0]}\le cW_{r,p},
\]
and the proof is complete.
\end{proof}

\sect{Direct estimates: proof of Theorems~\ref{direct} and \ref{thm2direct}}\label{sec5}

The following lemma is \cite{sam}*{Corollary 4.4} with $r=0$.
\begin{lemma} \label{jacklemma}
Let $k \in\N$,    $ \a\geq 0$, $ \b\geq 0$  and $f\in\Lpab$, $0< p \leq \infty$.
 Then, there exists $N\in\N$ depending on $k$,   $p$, $\a$ and $\b$, such that
 for every $n \geq N$ and $0<\vartheta \leq 1$,  there is a polynomial $P_n \in\Poly_n$ satisfying
\[
\norm{\wab  (f -P_n )}{p}  \leq c   \w_{k,0}^{*\varphi}(f , \vartheta/n)_{\a,\b,p}
\leq c   \w_{k, 0}^\varphi(f ,   \vartheta/n)_{\a,\b,p},
\]
and
\[
n^{-k} \norm{\wab \varphi^{k} P_n^{(k)}}{p}  \leq c \w_{k,0}^{*\varphi}(f, \vartheta/n)_{\a,\b,p}
\leq c   \w_{k, 0}^\varphi(f,  \vartheta/n)_{\a,\b,p}  ,
\]
where constants $c$  depend only on  $k$,  $p$, $\a$, $\b$ and $\vartheta$.
\end{lemma}

\begin{proof}[Proof of $\thm{direct}$]
Estimate \ineq{dir} immediately follows from \lem{jacklemma} for $n\ge N$. For $k\le n< N$, \ineq{dir} follows from  \cor{whcor} with $\theta:=1/N$, since
\[
E_n(f)_{\a,\b,p}\le  E_k(f)_{\a,\b,p}\le  c \w_{k,0}^\varphi (f,1/N)_{\a,\b,p} \le     c \w_{k,0}^\varphi (f,1/n)_{\a,\b,p}.
\]
\end{proof}

\begin{remark}
In the case $1\le p \le \infty$, it was shown by Ky \cite{ky}*{Theorem 4} $($see also Luther and Russo \cite{lr}$)$ that if $\a,\b\ge0$, then
\be \label{ky}
E_n(f)_{w_{\a,\b},p}\le c  \omega^k_\varphi(f,1/n)_{w_{\a,\b},p}\,,\quad n\ge n_0.
\ee
By virtue of \cite{sam}*{(2.2)}, we have, for $1\le p \le \infty$,
\be\label{claim}
\wkr(f^{(r)},t)_{\a,\b,p}\sim\omega^k_\varphi( f^{(r)},t)_{w_{\a,\b}\varphi^r,p}, \quad 0<t\leq t_0.
\ee
Thus, in the case $1\le p \le \infty$, \ineq{dir} with $n\geq n_0$ follows from \ineq{ky}.
We also remark that, even though \ineq{ky} was stated with $n_0=k$ in \cite{ky}, the proof used \cite{dt}*{Theorem 6.1.1} where $0<t\leq t_0$, and so was only justified for sufficiently large $n$.
\end{remark}

\begin{proof}[Proof of $\thm{thm2direct}$]
The case $r=0$ is \thm{direct}. Thus we may assume that $r\geq 1$. It follows by \cite{dt}*{Theorem 8.2.1 and (6.3.2)} that, for $n\ge n_0$,
\begin{align}\label{estim}
E_n(f)_{w_{\a,\b},p}&\le c\int_0^{1/n}(\Omega^{k+r}_\varphi(f,t)_{w_{\a,\b},p}/t)dt\\&\le c\int_0^{1/n}t^{r-1}\Omega^k_\varphi(f^{(r)},t)_{w_{\a,\b}\varphi^r,p}\,dt\nonumber\\&\le\frac c{n^r}\Omega^k_\varphi(f^{(r)},1/n)_{w_{\a,\b}\varphi^r,p}\le\frac c{n^r}\omega^k_\varphi(f^{(r)},1/n)_{w_{\a,\b}\varphi^r,p}\,,\nonumber
\end{align}
and so \ineq{dir} follows by \ineq{claim}.
For $k+r\le n < n_0$, \ineq{dir} immediately follows from \thm{whitneythm} with $\theta:=1/n_0$, as above.
This completes the proof.
\end{proof}

\sect{Inverse theorem: proof of \thm{inverse}}\label{sec4}

We first prove this theorem in the case $r\ge 1$.

For the proof we need the following fundamental inequality   (see  \cites{kha, pot} as well as \cite{dt}*{(8.1.3)}):  given $\a,\b\in J_p$, $1\le p\le\infty$, we have
\be\label{pt}
\norm{ \wab \varphi^rp^{(r)}_n}{p}\le c(r,p, \a, \b)n^r \norm{\wab p_n}{p},\quad p_n\in\Pn.
\ee
Let $f\in \Lpab$ and let $P_n\in\Pn$ be a polynomial of   best approximation of $f$ in $\Lpab$. Then
$E_n(f)_{w_{\a,\b},p} = \|\wab(f-P_n)\|_{p}$, $n\ge1$. 

Throughout the proof, we often use the estimate
\begin{align} \label{twoj}
\lefteqn{ \sum_{j=l}^m ( 2^j N)^\nu E_{2^jN}(f)_{w_{\a,\b},p}}\\ \nonumber
 &\leq
(1+2^\nu)  \sum_{j=l}^{m-1} ( 2^j N)^\nu E_{2^jN}(f)_{w_{\a,\b},p}
 \\ \nonumber
 &\leq
(1+2^\nu) 2^\nu \sum_{j=l}^{m-1}    \sum_{n=2^{j-1}N +1}^{2^{j}N } n^{\nu-1}   E_{n}(f)_{w_{\a,\b},p}\\ \nonumber
 &=
(1+2^\nu) 2^\nu \sum_{n=2^{l-1} N +1}^{2^{m-1} N} n^{\nu-1}   E_{n}(f)_{w_{\a,\b},p},
\end{align}
where $\nu\geq 1$ and $1\leq l < m$, which is also valid if $m=\infty$.

We  represent $f$ as the telescopic series
\be \label{teles}
f=P_{k+r}+(P_N-P_{k+r})+\sum_{j=0}^\infty\left(P_{2^{j+1}N}-P_{2^j N}\right)=:P_{k+r}+Q+\sum_{j=0}^\infty Q_j .
\ee
Since
\be \label{intineq}
\|\wab Q_j\|_{p}\le\left\|\wab(P_{2^{j+1}N}-f)\right\|_{p}+\left\|\wab(f-P_{2^jN})\right\|_{p}\le cE_{2^jN}(f)_{\wab,p},
\ee
we have by virtue of  \ineq{pt} and \ineq{rcond}, for each $1\le\nu\le r$,
\begin{eqnarray*}
\sum_{j=0}^\infty \norm{\wab \varphi^{\nu}Q_j^{(\nu)} } {p} &\le& c\sum_{j=0}^\infty ( 2^{j+1}N)^\nu E_{2^jN}(f)_{w_{\a,\b},p}\\
&\leq &
c N^\nu E_{ N}(f)_{w_{\a,\b},p} + c \sum_{n=N+1}^\infty  n^{\nu-1}  E_{n}(f)_{w_{\a,\b},p}
<\infty.
\end{eqnarray*}

By the same argument as in the proof of \cite{kls}*{Theorem 9.1},
it follows that almost everywhere $f(x)$ is identical with a function possessing an absolutely continuous derivative of order $(r-1)$ and $f^{(r)} \in \Lp[-1+\e, 1-\e]$, for any $\e>0$.
In particular, differentiation of \ineq{teles} is justified, and $f\in \B_p^r(\wab)$.

By \cite{sam}*{Lemma 4.1}, since $r/2+\a\geq 0$ and   $r/2+\b \geq 0$,   we have
\[
\wkr(Q_j^{(r)},t)_{\a,\b,\,p}\le c\norm{ \wab \varphi^rQ_j^{(r)} }{p}
\]
and
\[
\wkr(Q_j^{(r)},t)_{\a,\b,\,p}\le ct^k \norm{\wab \varphi^{r+k}Q_j^{(r+k)} }{p}.
\]
Hence, by \ineq{pt} and \ineq{intineq} we obtain
\[
\wkr\bigl(Q_j^{(r)},t\bigr)_{\a,\b,p}\le c( 2^{j+1} N)^r \norm{\wab Q_j}{p}\le c (2^jN)^r    E_{ 2^j N}(f)_{\wab,p}
\]
and
\[
\wkr\bigl(Q_j^{(r)},t\bigr)_{\a,\b,p}\le ct^k(2^{j+1} N)^{r+k} \norm{
\wab Q_j}{p}\le c t^k  (2^j N)^{r+k}     E_{2^j N}(f)_{\wab,p}.
\]

Denoting $J:=\min\{j\in\N_0: 2^{-j} \le N t\}$ (note that  $2^{-J} \le Nt < 2^{-J+1}$ if $J\geq 1$, and   $Nt\geq 1$ if $J=0$)
we now have by \ineq{twoj}
\begin{eqnarray}\label{app1}
\wkr\biggl(\sum_{j=J+1}^\infty Q_j^{(r)},t\biggr)_{\a,\b,p}
&\le&c\sum_{j=J+1}^\infty\wkr\bigl(Q_j^{(r)},t\bigr)_{\a,\b,p}\\
&\le&c\sum_{j=J+1}^\infty (2^j N)^r  E_{2^jN}(f)_{w_{\a,\b},p}\nonumber\\
&\le&c\sum_{n=2^{J} N+1}^\infty n^{r-1}E_n(f)_{w_{\a,\b},p}\nonumber\\ \nonumber
&\le&
c \sum_{n>\max\{N,1/t\}} n^{r-1}E_n(f)_{w_{\a,\b},p},
\end{eqnarray}
since $2^J N+1 > \max\{N,1/t\}$.
Now,
  if $J\geq 2$, then \ineq{twoj} implies
\begin{eqnarray} \label{app2}
\lefteqn{ \wkr\biggl( \sum_{j=0}^J Q_j^{(r)},t\biggr)_{\a,\b,p}
 \le ct^k\sum_{j=0}^J  (2^j N)^{r+k}      E_{2^j N}(f)_{w_{\a,\b},p} }\\ \nonumber
&\le& c t^k N^{r+k} E_{N}(f)_{w_{\a,\b},p} +  ct^k\sum_{n=N+1}^{2^{J-1} N}n^{r+k-1}E_n(f)_{w_{\a,\b},p} \\ \nonumber
&\le&
c t^k N^{r+k} E_{N}(f)_{w_{\a,\b},p} +
ct^{k}\sum_{N \leq n\leq\max\{N,1/t\}}n^{k+r-1}E_n(f)_{w_{\a,\b},p} ,
\end{eqnarray}
where we used the fact that  $2^{J-1}N \leq \max\{N,1/t\}$.
If $J=0$ or $1$, then we have
\[ 
\wkr\biggl( \sum_{j=0}^J  Q_j^{(r)},t\biggr)_{\a,\b,p} \le  ct^k  N^{r+k}      E_{N}(f)_{w_{\a,\b},p}, 
\]
and so the last estimate in \ineq{app2} is valid in this case as well.

Finally, if $N\geq k+r$, then
\begin{align}\label{app3}
\wkr(P_{k+r}^{(r)}+ Q^{(r)}, t)_{\a,\b,p} &=
\wkr(Q^{(r)}, t)_{\a,\b,p} \le ct^k \norm{\wab \varphi^{k+r}Q^{(k+r)} }{p} \\ & \le ct^k N^{r+k} \norm{ \wab Q }{p}
 \le ct^k N^{r+k}E_{k+r}(f)_{w_{\a,\b},p},\nonumber
\end{align}
and if $N<k+r$, then $\wkr(P_{k+r}^{(r)}+ Q^{(r)}, t)_{\a,\b,p} = 0$, so that we don't need \ineq{app3}.

Combining \ineq{app1}-\ineq{app3} and using the fact that, if $N\geq k+r$, then $E_{N}(f)_{w_{\a,\b},p} \leq E_{k+r}(f)_{w_{\a,\b},p}$ and, if $N<k+r$, then the first term in the last inequality in \ineq{app2} can be absorbed by the second term in that inequality,
we obtain \ineq{inverse1}, and our proof is complete in the case $r\geq 1$.

Suppose now that $r=0$. We  represent $f$ as
\be \label{teleszero}
f= 
 P_{k}+Q+\sum_{j=0}^J Q_j + \left(f-P_{2^{J+1}N}\right) ,
\ee
where $Q:= P_N-P_{k}$ and $Q_j := P_{2^{j+1}N}-P_{2^j N}$, and estimate the last term.
We have
\[
\norm{\wab(f-P_{2^{J+1}N})}{p} \leq c E_{ 2^{J+1} N}(f)_{\wab,p},
\]
and in the case $J=0$ or $1$, we use the fact that $Nt\geq c$, to conclude
\[
E_{ 2^{J+1} N}(f)_{\wab,p} \le E_{N}(f)_{\wab,p} = N^k t^k (Nt)^{-k} E_{N}(f)_{\wab,p} \le c(N) t^k  E_N(f)_{\wab,p}.
\]
If $J\ge 2$, we recall that $2^{J-1} N < 1/t \le 2^{J} N$, so that $\max\{N,1/t\} = 1/t$, and write
\begin{eqnarray*}
E_{ 2^{J+1} N}(f)_{\wab,p} &\le& (2^{J-2}N)^{-1}\sum_{n=2^{J-2} N+1}^{2^{J-1} N} E_{n}(f)_{\wab,p} \\
&\le&
(2^{J-2}N)^{-k}  \sum_{n=2^{J-2} N+1}^{2^{J-1} N} n^{k-1} E_{n}(f)_{\wab,p}\\
&\le&
4^k t^k \sum_{N\le n <1/t} n^{k-1} E_{n}(f)_{\wab,p}.
\end{eqnarray*}
It now remains to apply  \ineq{app2} and \ineq{app3} with $r=0$, in order to complete the proof of \ineq{inverse1} in the case $r=0$.
\qed

\sect{Appendix}

The following sharp Marchaud inequality was proved in \cite{dd}.

\begin{theorem}[\mbox{\cite{dd}*{Theorem 7.5}}]
For $m \in\N$, $1<p<\infty$ and $\a,\b\in J_p$,  we have
\[
K_{m,\varphi}(f,t^m)_{\wab,p} \leq C t^m \left( \int_t^1 \frac{K_{m+1,\varphi}(f,u^{m+1})_{\wab,p}^q}{u^{mq+1}}\, du + E_m(f)_{\wab,p}^q \right)^{1/q}
\]
and
\[
K_{m,\varphi}(f,t^m)_{\wab,p} \leq C t^m \left( \sum_{n <1/t} n^{qm-1} E_n(f)_{\wab,p}^q \right)^{1/q} ,
\]
where $q = \min(2,p)$.
\end{theorem}

\begin{corollary}
For   $1<p<\infty$, $r\in\N_0$, $m \in\N$, $r/2+\a\geq 0$, $r/2+\b\ge 0$ and $f\in\B^r_p(\wab)$, we have
\[
\omega^{\varphi}_{m,r} (f^{(r)},t)_{\a,\b,p} \leq C t^m \left( \int_t^1\frac{\omega^{\varphi}_{m+1,r}(f^{(r)},u)_{\a,\b, p}^q}{u^{mq+1}}\, du + E_m(f^{(r)})_{\wab\varphi^r,p}^q \right)^{1/q}
\]
and
\[
\omega^{\varphi}_{m,r}(f^{(r)},t)_{\a,\b, p} \leq C t^m \left( \sum_{n <1/t} n^{qm-1} E_n(f^{(r)})_{\wab \varphi^r,p}^q \right)^{1/q} ,
\]
where $q = \min(2,p)$.
\end{corollary}

The following sharp Jackson inequality was proved in \cite{ddt}.

\begin{theorem}[\mbox{\cite{ddt}*{Theorem 6.2}}]
For $1<p<\infty$, $\a, \b \in J_p$ and $m \in\N$, we have
\[
2^{-nm} \left( \sum_{j=j_0}^n 2^{mjs} E_{2^j}(f)^s_{\wab,p} \right)^{1/s} \leq C K_{m,\varphi}(f, 2^{-nm})_{\wab,p}
\]
and
\[
2^{-nm} \left( \sum_{j=j_0}^n 2^{mjs}  K_{m+1,\varphi}(f, 2^{-j(m+1)})_{\wab,p}^s    \right)^{1/s} \leq C K_{m,\varphi}(f, 2^{-nm})_{\wab,p} ,
\]
where $2^{j_0} \geq m$ and $s = \max(p,2)$.
\end{theorem}

\begin{corollary}
For   $1<p<\infty$, $r\in\N_0$, $m \in\N$, $r/2+\a\ge 0$, $r/2+\b\ge 0$ and $f\in\B^r_p(\wab)$, we have
\[
2^{-nm} \left( \sum_{j=j_0}^n 2^{mjs} E_{2^j}(f^{(r)})^s_{\wab\varphi^r,p} \right)^{1/s} \leq C  \omega^{\varphi}_{m,r}(f^{(r)},2^{-n})_{\a,\b, p}
\]
and
\[
2^{-nm} \left( \sum_{j=j_0}^n 2^{mjs}  \omega^{\varphi}_{m+1,r}(f^{(r)},2^{-j})_{\a,\b, p}^s  \right)^{1/s} \leq C \omega^{\varphi}_{m,r}(f^{(r)},2^{-n})_{\a,\b, p},
\]
where $2^{j_0} \geq m$ and $s = \max(p,2)$.
\end{corollary}

\begin{corollary}
For   $1<p<\infty$, $r\in\N_0$, $m \in\N$, $r/2+\a\ge 0$, $r/2+\b\ge 0$ and $f\in\B^r_p(\wab)$, we have
\[
t^m  \left(  \int_t^{1/m}        \frac{\omega^{\varphi}_{m+1,r}(f^{(r)},u)_{\a,\b,p}^s}{u^{ms+1}} \, du \right)^{1/s} \leq C \omega^{\varphi}_{m,r}(f^{(r)},t)_{\a,\b,p}, \quad 0<t\leq 1/m,
\]
where  $s = \max(p,2)$.
\end{corollary}

\end{document}